\documentclass{amsart}

\usepackage{amsmath, amssymb,amsfonts,amsthm}
\usepackage{tikz}
\usepackage{textcomp}
\usepackage{seqsplit}
\usepackage{setspace}
\usepackage{listings}
\usepackage{marvosym}
\usepackage[top=1.9in, bottom=1.9in, left=2in, right=2in]{geometry}
\usepackage{nopageno} 
\usepackage{mathdots}
\usepackage{enumitem}
\usepackage{colortbl}
\usepackage{graphicx}

\newcommand{\Z}{{\mathbb{Z}}}

\newcommand{\N}{{\mathbb{N}}}
\newcommand{\floor}[1]{\left \lfloor{#1}\right \rfloor}

\newtheorem{thm}{Theorem}[section]

\newtheorem{lem}[thm]{Lemma}
\newtheorem{cor}[thm]{Corollary}

\theoremstyle{definition}
\newtheorem{definition}[thm]{Definition}
\newtheorem{example}[thm]{Example}
\newtheorem{example*}{Example}

\theoremstyle{remark}

\begin{document}

\title{Bounds on Zimin Word Avoidance}

\author{Joshua Cooper*}
\author{Danny Rorabaugh*}
\let\thefootnote\relax\footnote{*University of South Carolina}

\begin{abstract}
	How long can a word be that avoids the unavoidable? Word $W$ encounters word $V$ provided there is a homomorphism $\phi$ defined by mapping letters to nonempty words such that $\phi(V)$ is a subword of $W$. Otherwise, $W$ is said to avoid $V$. If, on any arbitrary finite alphabet, there are finitely many words that avoid $V$, then we say $V$ is unavoidable. Zimin (1982) proved that every unavoidable word is encountered by some word $Z_n$, defined by: $Z_1 = x_1$ and $Z_{n+1} = Z_nx_{n+1}Z_n$. Here we explore bounds on how long words can be and still avoid the unavoidable Zimin words.
\end{abstract}

\maketitle

\thispagestyle{empty}

In 1929, Frank Ramsey proved that, for any fixed $r,n,\mu \in \Z^+$, every sufficiently large set $\Gamma$ with its $r$-subsets partitioned into $\mu$ classes is guaranteed to have a subset $\Delta_n \subseteq \Gamma$ such that all the $r$-subsets of $\Delta_n$ are in the same class \cite{Ramsey}. This was the advent of a major branch of combinatorics that became known as Ramsey theory. Often applied to graph theoretic structures, Ramsey theory looks at how large a random structure must be to guarantee that a given substructure exists or a given property is satisfied. Here we apply this paradigm to an existence result from the combinatorics of words.

\begin{definition}
	A {\it $q$-ary word} is a string of characters, at most $q$ of them distinct.
\end{definition}

Over a fixed $q$-letter alphabet, the set of all finite words forms a semigroup with concatenation as the binary operation (written multiplicatively) and the empty word $\varepsilon$ as the identity element. We also have a binary subword relation $\leq$ where $V \leq W$ when $W = UVU'$ for some words $U$, $V$, and $U'$. That is, $V$ appears contiguously in $W$.
	
\begin{definition}
	We call word $W$ an {\it instance} of $V$ provided
	\begin{itemize}[leftmargin=.36in]
		\item $V = x_0x_1 \cdots x_{m-1}$ where each $x_i$ is a letter;
		\item $W = A_0A_1 \cdots A_{m-1}$ with each $A_i \neq \varepsilon$ and $A_i = A_j$ whenever $x_i=x_j$.
	\end{itemize}
	
	Equivalently, $W$ is a {\it $V$-instance} provided there exists some semigroup homomorphism $\phi$ such that $\phi(x_i) = A_i \neq \varepsilon$ for each $i$.
\end{definition}

\begin{example}
$W = abbcabbxdc$ is an instance of $V = xyxzy$, with $\phi$ defined by $\phi(x) = abb$, $\phi(y)=c$, and $\phi(z)=xd$.
\end{example}

\begin{definition}
	A word $U$ {\it encounters} word $V$ provided some subword $W \leq U$ is an instance of $V$. If $U$ fails to encounter $V$, then $U$ {\it avoids} $V$.
\end{definition}

\begin{figure}[h]
	\caption {Binary words that avoid $xx$.}
	\bigskip
	\begin{tikzpicture}
		\draw (0,0) node{$\varepsilon$};
		\draw (-1,.7) node[above]{$a$}--(0,.3)--(1,.7) node[above]{$b$};
		\draw (-1.5,1.7) node[above]{$aa$}--(-1,1.3)--(-.6,1.7) node[above]{$ab$};
		\draw (.6,1.7) node[above]{$ba$}--(1,1.3)--(1.5,1.7) node[above]{$bb$};
		\draw (-1.2,1.7)--(-1.8,2.2);
		\draw (-1.2,2.2)--(-1.8,1.7);
		\draw (1.2,1.7)--(1.8,2.2);
		\draw (1.2,2.2)--(1.8,1.7);
		\draw (-1.2,2.7) node[above]{$aba$}--(-.6,2.3)--(-.4,2.7) node[above]{$abb$};
		\draw (1.2,2.7) node[above]{$bab$}--(.6,2.3)--(.4,2.7) node[above]{$baa$};
		\draw (-.1,2.7)--(-.7,3.2);
		\draw (-.1,3.2)--(-.7,2.7);
		\draw (.1,2.7)--(.7,3.2);
		\draw (.1,3.2)--(.7,2.7);
		\draw(-1.7,3.7) node[above]{$abaa$}--(-1.2,3.2)--(-.6,3.7) node[above]{$abab$};
		\draw(1.7,3.7) node[above]{$babb$}--(1.2,3.2)--(.6,3.7) node[above]{$baba$};
		\draw (-2,3.7)--(-1.4,4.2);
		\draw (-1.4,3.7)--(-2,4.2);
		\draw (2,3.7)--(1.4,4.2);
		\draw (1.4,3.7)--(2,4.2);
		\draw (-.9,3.7)--(-.3,4.2);
		\draw (-.3,3.7)--(-.9,4.2);
		\draw (.9,3.7)--(.3,4.2);
		\draw (.3,3.7)--(.9,4.2);
	\end{tikzpicture}
\end{figure}
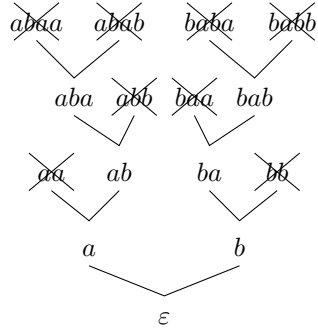

We see in Figure 1 that $xx$ is avoided by only finitely many words over a two-letter alphabet. However, it has been known for over a century \cite{Thue} that $xx$ can be avoided by arbitrarily long (even infinite) ternary words.

\begin{definition}
	A word $V$ is {\it unavoidable} provided for any finite alphabet, there are only finitely many words that avoid $V$.
\end{definition}

A. I. Zimin proved an elegant classification of all unavoidable words \cite{Zimin}.

\begin{definition}		
	Define the {\it $n^{th}$ Zimin word} recursively by $Z_0 := \varepsilon$ and, for $n \in \N$, $Z_{n+1} := Z_nx_nZ_n$. Using the alphabet rather than indexed variables:
	$$Z_1 = a, \quad Z_2 = a{\bf b}a, \quad Z_3 = aba{\bf c}aba, \quad Z_4 = abacaba{\bf d}abacaba, \quad \ldots$$

Equivalently, $Z_n$ can be defined over the natural numbers as the word of length $2^n-1$ such that the $i^{\text{th}}$ letter is the 2-adic order of $i$ for $1 \leq i < 2^n$.
\end{definition}
	
\begin{thm}[Zimin, 1982]
	A word $V$ with $n$ distinct letters is unavoidable if and only if $Z_n$ encounters $V$.
\end{thm}

\bigskip

\section{Avoiding the Unavoidable}


From Zimin's explicit classification of unavoidable words, a natural question arises in the Ramsey theory paradigm: for a fixed unavoidable word $V$, how long can a word be that avoids $V$?
Our approach to this question is to start with avoiding the Zimin words, which gives upper bounds for all unavoidable words.
Define $f(n,q)$ to be the smallest integer $M$ such that every $q$-ary word of length $M$ encounters $Z_n$.

\begin{thm}
\label{upper}
	For $n,q \in \Z^+$ and $Q := 2q+1$,
	$$f(n,q) \leq {}^{n-1}Q := Q^{Q^{^{\iddots^{{}_Q}}}},$$
	with $Q$ occurring $n-1$ times in the exponential tower.
\end{thm}

\begin{proof}
	We proceed via induction on $n$. For the base case, set $n=1$. Every nonempty word is an instance of $Z_1$, so $f(1,q) = 1$.

	For the inductive hypothesis, assume the claim is true for some positive $n$ and set $T:=f(n,q)$. That is, every $q$-ary word of length $T$ encounters $Z_n$.  Concatenate any $q^T+1$ strings $W_0, W_1, \ldots, W_{q^T}$ of length $T$ with an arbitrary letter $a_i$ between $W_{i-1}$ and $W_{i}$ for each positive $i \leq q^T$:
	$$U := W_0 \; a_1\; W_1 \; a_2 \; W_2 \; a_3\; \cdots \; W_{q^T-1} \; a_{q^T} \; W_{q^T}.$$
	
	By the pigeonhole principle, $W_i = W_j$ for some $i < j$.
	That string, being length $T$, encounters $Z_n$.  
	Therefore, we have some word $W \leq W_i$ that is an instance of $Z_n$ and shows up twice, disjointly, in $U$.
	The extra letter $a_{i+1}$ guarantee that the two occurrences of $W$ are not consecutive.
	This proves that an arbitrary word of length $(T+1)(q^T + 1)-1$ witnesses $Z_{n+1}$, so
	$$f(n+1,q) \leq (T+1)(q^T + 1)-1  \leq (2q+1)^T = Q^T.$$
\end{proof}

There is clearly a function $Q(n,q)$ such that $f(n+1,q) \leq {Q(n,q)}^{f(n,q)}$ and $Q(n,q)\rightarrow q$ as $n\rightarrow \infty$. 
No effort has been made to optimize the choice of function, as such does not decrease the tetration in the bound.

The technique used to prove Theorem \ref{upper} is first found in Lothaire's proof of unavoidability of $Z_n$ (\cite{Lothaire}, 3.1.3). The technique in Zimin's original proof \cite{Zimin} implicitly gives that for $n\geq 2$,
$$f(n+1,q+1) \leq (f(n+1,q)+2|Z_{n+1}|)f(n,|Z_{n+1}|^2q^{f(n+1,q)}).$$
This is an Ackermann-type function for an upper bound, which is much larger than the primitive recursive bound from Theorem \ref{upper}.

Table 1 shows known values of $f(n,2)$. 
Supporting word-lists and Sage code are found in the Appendix.
\begin{table}[h]
	\caption{Values of $f(n,2)$ for $n\leq 4$.}
	\begin{tabular}{c|c|c}
		$n$ & $Z_n$ & $f(n,2)$\\ \hline
		0 & $\varepsilon$ & 0\\
		1 & a & 1\\
		2 & aba & 5\\
		3 & abacaba & 29\\
		4 & abacabadabacaba & $\geq 10483$
	\end{tabular}\\
\end{table}

\section{Finding a Lower Bound with the First Moment Method}

Throughout this section, $q$ is a fixed integer greater than 1. Given a fixed alphabet of $q$ letters, $C(n,q,M)$ denotes the set of length-$M$ instances of $Z_n$. That is
$$C(n,q,M) := \{W \mid W \in \{x_0,\ldots,x_{q-1}\}^M \text{ is a }Z_n\text{-instance}\}.$$

\begin{lem}
	For all $n,M \in \Z^+$,
	$$|C(n,q,M+1)| \geq q\cdot|C(n,q,M)|.$$
\end{lem}

\begin{proof}
	Take arbitrary $W \in C(n,q,M)$. 
	We can write $W = W_1W_0W_1$ with $W_1 \in C(n-1,q,N)$, where $2N < M$. 
	Choose the decomposition of $W$ to minimize $|W_1|$. 
	Then $W_1W_0x_iW_1 \in C(n,q,M+1)$ for each $i < q$. 
	
	The lemma follows, unless a $Z_n$-instance of length $M+1$ can be generated in two ways -- that is, if $W_1W_0aW_1 = V_1V_0bV_1$ for some $V_1V_0V_1 = V$, where $|V_1|$ is also minimized. 
	If $|V_1|<|W_1|$, then $V_1$ is a prefix and suffix of $W_1$, so $|W_1|$ was not minimized.
	But if $|V_1|>|W_1|$, then $W_1$ is a prefix and suffix of $V_1$, so $|V_1|$ was not minimized.
	Therefore, $|V_1|=|W_1|$, so $V_1=W_1$, which implies $a=b$ and $V=W$.
\end{proof}

\begin{cor}[Monotonicity] For all $n,M \in \Z^+$,
	\begin{eqnarray*}
		\Pr\left(W \in C(n,q,M+1) \mid W\in \{x_0,\ldots,x_{q-1}\}^{M+1}\right) \quad\\
		\geq \Pr\left(W \in C(n,q,M) \mid W\in \{x_0,\ldots,x_{q-1}\}^M\right),
	\end{eqnarray*}
	assuming uniform probability on words of a fixed length.
\end{cor}

\begin{lem}
	For all $n,M \in \Z^+$,
	$$|C(n,q,M)| \leq \left(\frac{q}{q-1}\right)^{n-1}q^{(M-2^n+n+1)}.$$
\end{lem}

\begin{proof}
	The proof proceeds by induction on $n$. For the base case, set $n=1$. Every non-empty word is an instance of $Z_1$, so $|C(1,q,M)| = q^M$.
	
	For the inductive hypothesis, assume the claim is true for some positive $n$. The first inequality below derives from the following way to overcount the number of $Z_{n+1}$-instances of length $M$. Every such word can be written as $UVU$ where $U$ is a $Z_n$-instance of length $j<M/2$. Since an instance of $Z_n$ can be no shorter than $Z_n$, we have $2^n-1 \leq j < M/2$. For each possible $j$, there are $|C(n,q,j)|$ ways to choose $U$ and $q^{M-2j}$ ways to choose $V$. This is an overcount, since a Zimin-instance may have multiple decompositions.

	\begin{eqnarray*}
		|C(n+1,q,M)| & \leq & \sum_{j = 2^n-1}^{\floor{(M-1)/2}} |C(n,q,j)|q^{M-2j} \\
		& \leq & \sum_{j = 2^n-1}^{\floor{(M-1)/2}}\left(\frac{q}{q-1}\right)^{n-1}q^{(j-2^n+n+1)}q^{M-2j}\\
		& = & \left(\frac{q}{q-1}\right)^{n-1}q^{(M-2^n+n+1)} \sum_{j = 2^n-1}^{\floor{(M-1)/2}}q^{-j}\\
		& < & \left(\frac{q}{q-1}\right)^{n-1}q^{(M-2^n+n+1)} \sum_{j = 2^n-1}^{\infty}q^{-j}\\
		& = & \left(\frac{q}{q-1}\right)^{n-1}q^{(M-2^n+n+1)}\left(\frac{q^{-(2^n-1)+1}}{q-1}\right)\\
		& = & \left(\frac{q}{q-1}\right)^{(n-1)+1}q^{(M -2^{n+1}+ (n+1)+1)}.
	\end{eqnarray*}
\end{proof}

\begin{cor}	
	For all $n,M \in \Z^+$,
	$$\Pr\left(W \in C(n,q,M) \mid W\in \{x_0,\ldots,x_{q-1}\}^M\right) \leq \left(\frac{q}{q-1}\right)^{n-1}q^{(-2^n+n+1)},$$
	assuming uniform probability on words of length $M$.
\end{cor}

\begin{thm}
	$$f(n,q) \geq q^{2^{(n-1)}(1+o(1))} \quad (q\rightarrow \infty, n \rightarrow \infty).$$
\end{thm}

\begin{proof}
	Let word $W$ consist of $M$ uniform, independent random selections from the alphabet $\{x_0, \ldots, x_{q-1}\}$.
	Define the random variable $X$ to count the number of subwords of $W$ that are instances of $Z_n$ (including repetition if a single subword occurs multiple times in $W$):
	$$X = |\{V \mid W \geq V \in C(n,q,|V|)\}|.$$
	By monotonicity with respect to word length:
	\begin{eqnarray*}
		E(X) & \leq & |\{V \mid V \leq W\}| \cdot \Pr(W \in C(n,q,M))\\	
		& \leq & \binom{M+1}{2} \left(\frac{q}{q-1}\right)^{n-1}q^{(-2^n+n+1)}\\	
		& < & M^2 e^{(n-1)/(q-1)}q^{(-2^n+n+1)}.
	\end{eqnarray*}
	
	There exists a word of length $M$ that avoids $Z_n$ when $E(X) < 1$.
	It suffices to show that:
	$$M^2 \left(e^{(n-1)/(q-1)}q^{(-2^n+n+1)}\right) \leq 1.$$
	
	Solving for $M$:
	\begin{eqnarray*}
		M &\leq&  \left(e^{(n-1)/(q-1)}q^{(-2^n+n+1)}\right)^{-1/2}\\
		&=& q^{2^{(n-1)}}\left(e^{(n-1)/(q-1)}q^{(n+1)}\right)^{-1/2}\\
		&=& q^{2^{(n-1)}(1+o(1))}.
	\end{eqnarray*}
\end{proof}


\section*{Continuing work}

Current efforts to improve bounds on the probability that a word is an instance of $Z_n$ will help close the gap between the lower and upper bounds on $f(n,q).$

\section*{Acknowledgements}

We would like to express our gratitude to Professor George McNulty for his algebraic contributions to the field of word avoidability \cite{McNulty}, and for bringing this topic to our attention. Also, to the organizers of the {\it 45th Southeast International Conference on Combinatorics, Graph Theory, and Computing}, who accepted this paper for presentation on 04 March 2014. And finally, to the participants of the Graduate Discrete Mathematics Seminar at the University of South Carolina, who formed a wonderful audience for early presentations of this work.



\newpage

\section*{Appendix: Binary Words that Avoid $Z_n$}

\subsection*{All binary words that avoid $Z_2$.} \text{}\\
The following 13 words are the only words over the alphabet $\{0,1\}$ that avoid $Z_2 = aba$. 

$$\begin{matrix}
	\varepsilon, & 0, & 00, & 001, & 0011, \\
	& & 01, & 011, & \\
	& 1, & 10, & 100, & \\
	& & 11, & 110, & 1100. 
\end{matrix}$$

\vspace{.5cm}

\subsection*{Maximum-length binary words that avoid $Z_3$.}\text{}\\
The following 48 words are the only words of length $f(3,2)-1 = 28$ over the alphabet $\{0,1\}$ that avoid $Z_3= abacaba$. All binary words of length $f(3,2) = 29$ or longer encounter $Z_3$. This result is easily, computationally verified by constructing the binary tree of words on $\{0,1\}$, eliminating branches as you find words that encounter $Z_3$.\\

\begin{center}
	{\small
		\begin{tabular}{c c}
			\begin{tabular}{c}
				0010010011011011111100000011,\\
				0010010011111100000011011011,\\
				0010010011111101101100000011,\\
				0010101100110011111100000011,\\
				0010101111110000001100110011,\\
				0010101111110011001100000011,\\
				0011001100101011111100000011,\\
				0011001100111111000000101011,\\
				0011001100111111010100000011,\\
				0011011010010011111100000011,\\
				0011011011111100000010010011,\\
				0011011011111100100100000011,\\
				0011111100000010010011011011,\\
				0011111100000010101100110011,\\
				0011111100000011001100101011,\\
				0011111100000011011010010011,\\
				0011111100100100000011011011,\\
				0011111100100101101100000011,\\
				0011111100110011000000101011,\\
				0011111100110011010100000011,\\
				0011111101010000001100110011,\\
				0011111101010011001100000011,\\
				0011111101101100000010010011,\\
				0011111101101100100100000011,
			\end{tabular}
			&
			\begin{tabular}{c}
				1100000010010011011011111100,\\
				1100000010010011111101101100,\\
				1100000010101100110011111100,\\
				1100000010101111110011001100,\\
				1100000011001100101011111100,\\
				1100000011001100111111010100,\\
				1100000011011010010011111100,\\
				1100000011011011111100100100,\\
				1100000011111100100101101100,\\
				1100000011111100110011010100,\\
				1100000011111101010011001100,\\
				1100000011111101101100100100,\\
				1100100100000011011011111100,\\
				1100100100000011111101101100,\\
				1100100101101100000011111100,\\
				1100110011000000101011111100,\\
				1100110011000000111111010100,\\
				1100110011010100000011111100,\\
				1101010000001100110011111100,\\
				1101010000001111110011001100,\\
				1101010011001100000011111100,\\
				1101101100000010010011111100,\\
				1101101100000011111100100100,\\
				1101101100100100000011111100.
			\end{tabular}
		\end{tabular}
	}
\end{center}

\newpage

\subsection*{A long binary word that avoid $Z_4$:}\text{}\\
The following binary word of length 10482 avoids $Z_4 = abacabadabacaba$. This implies that $f(4,2)\geq 10483$. The word is presented here as an image with each row, consisting of 90 squares, read left to right. Each square, black or white, represents a bit. For example, the longest string of black in the first row is 14 bits long. We cannot have the same bit repeated $15 = |Z_4|$ times consecutively, as that would be a $Z_4$-instance. A string of 14 white bits is found in the 46th row.\\

\begin{center}
\includegraphics[width=270px]{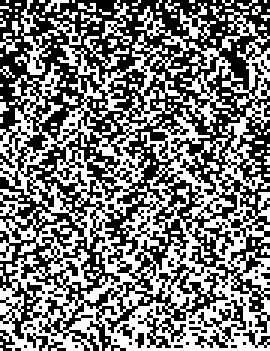}
\end{center}

\newpage

\subsection*{Verifying that a word avoids $Z_n$:}\text{}\\
The code to generate a $Z_4$-avoiding word of length 10482 is messy. The following, easy-to-validate, inefficient, brute-force, Sage \cite{Sage} code was used for verification of the word above. It took about half a day, running on an Intel\textregistered\,Core\texttrademark\,i5-2450M CPU $@$ 2.50GHz $\times$ 4. \\

\begin{lstlisting}[frame=single]
#Recursive function to test if V is an instance of Z_n
def inst(V,n):
	if n==1:
		if len(V)>0:
			return 1
		return 0
	else:
		top = ceil(len(V)/2)
		for i in range(2^(n-1)-1,top):
			if V[:i]==V[-i:]:
				if inst(V[:i],n-1):
					return 1
		return 0

#Paste word here as a string
W = 
L = len(W)
n = 4

#Check every subword V of length at least 2^n-1
for b in range(L+1):
	for a in range(b-(2^n-1)):
		if inst(W[a:b],n):
			print a,b,W[a:b]
\end{lstlisting}

\end{document}